\documentclass[11pt]{amsart}

\usepackage[T1]{fontenc}
\usepackage[latin1]{inputenc}
\usepackage[english]{babel}
\usepackage{amsmath,amssymb,amsthm} %amssymb loads amsfonts
\usepackage{enumitem}
\usepackage{hyperref}

\theoremstyle{plain}
\newtheorem{thm}{Theorem}[section]
\newtheorem{lem}[thm]{Lemma}

\newtheorem{cor}[thm]{Corollary}

\theoremstyle{definition}

\newtheorem*{acknowledgments}{Acknowledgements}

\theoremstyle{remark}
\newtheorem{rem}[thm]{Remark}

\newtheorem*{rem*}{Remark}
\newtheorem*{exa*}{Example}

\newcommand{\bB}{\mathbb B}
\newcommand{\bC}{\mathbb C}
\newcommand{\bD}{\mathbb D}

\newcommand{\bR}{\mathbb R}

\newcommand{\cB}{\mathcal B}
\newcommand{\cF}{\mathcal F}
\newcommand{\cH}{\mathcal H}

\DeclareMathOperator{\Mult}{Mult}

\newcommand{\tfa}{\text{ for all }}

\newcommand{\kerpart}[1]{\langle \cdot,#1 \rangle}
\newcommand{\res}[1]{\big|_{#1}}

\author{Michael Hartz}
\title[Nevanlinna-Pick spaces and hyponormality]{Nevanlinna-Pick spaces with hyponormal multiplication operators}
\address{Department of Pure Mathematics, University of Waterloo, Waterloo, ON N2L 3G1, Canada}
\email{mphartz@uwaterloo.ca}
\thanks{The author is partially supported by an Ontario Trillium Scholarship.}

\begin{document}

\begin{abstract}
  We show that the Hardy space on the unit disk is the only non-trivial irreducible reproducing kernel Hilbert
  space which satisfies the complete Nevanlinna-Pick property and hyponormality of all multiplication operators.
\end{abstract}

\maketitle

\section{Introduction}
\label{S:intro}

Let $\cH$ be a reproducing kernel Hilbert space on a set $X$ with kernel $K$.
In this short note, we study the relationship between two possible properties of $\cH$:
the complete Nevanlinna-Pick property and hyponormality of multiplication operators.
Recall that $\cH$ is said to be a Nevanlinna-Pick space if, given
$z_1,\ldots,z_n \in X$ and $w_1,\ldots,w_n \in \bC$, positivity of the matrix
\begin{equation*}
  \Big( (1-w_i \overline{w_j}) K(z_i,z_j) \Big)_{i,j=1}^n
\end{equation*}
is not only a necessary, but also a sufficient condition for the existence of a multiplier
$\varphi$ on $\cH$ of norm at most $1$ with
\begin{equation*}
  \varphi(z_i) = w_i \quad (i=1,\ldots,n). 
\end{equation*}
If the analogous result
for matrix-valued interpolation holds, then $\cH$ is called a complete Nevanlinna-Pick space (compare Chapter 5 in \cite{AM02}).
Spaces with this property have attracted a lot of attention, and it is known that they admit
appropriate versions of some classical theorems for the Hardy space $H^2$ on the disk,
such as the commutant lifting theorem \cite{BTV01} (see also \cite{AT02}), the Toeplitz-corona theorem \cite[Section 8.4]{AM02} and Beurling's theorem \cite[Section 8.5]{AM02}.

The second property we consider is hyponormality of multiplication operators, that is, the property that for every
multiplier $\varphi$ on $\cH$, the corresponding multiplication operator $M_\varphi \in \mathcal B(\cH)$
is hyponormal. While multiplication operators are not normal in typical examples,
they are subnormal and hence hyponormal for
a number of reproducing kernel Hilbert spaces, including Hardy and Bergman spaces on domains
in $\bC^d$.

Two results concerning weighted Hardy spaces serve as a motivation for the study of the relationship between the two properties.
Suppose for a moment that $\cH$ is a reproducing kernel Hilbert space on the open unit disk $\mathbb D$ with kernel $K$ of the form
\begin{equation*}
  K(z,w) = \sum_{n=0}^\infty a_n (z \overline {w})^n \quad (z,w \in \bD),
\end{equation*}
where $(a_n)$ is a sequence of positive numbers with $a_0 = 1$.
Note that the classical Hardy space $H^2$ corresponds
to the choice $a_n = 1$ for all $n$, in which case we recover the Szeg\H{o} kernel $(1-z \overline{w})^{-1}$.
We assume that multiplication by the coordinate
function $z$ induces a bounded multiplication operator $M_z$ on $\cH$.
Equivalently, the sequence $(a_n/a_{n+1})$ is bounded.
Then the operator $M_z$ is hyponormal if and only if
\begin{equation*}
  \frac{a_n}{a_{n-1}} \ge \frac{a_{n+1}}{a_n} \quad \text{ for all } n \ge 1
\end{equation*}
(see Section 7 in \cite{Shields74}, and note that the sequence $(\beta(n))$ there is related to $(a_n)$ via 
$a_n = \beta(n)^{-2}$).
On the other hand, a sufficient condition for $\cH$ being a complete Nevanlinna-Pick space is that the reverse inequalities
\begin{equation*}
  \frac{a_n}{a_{n-1}} \le \frac{a_{n+1}}{a_n} \quad \text{ for all } n \ge 1
\end{equation*}
hold (see Lemma 7.38 and Theorem 7.33 in \cite{AM02}).
Since this condition is not necessary, the two results do not immediately tell us anything new about
weighted Hardy spaces satisfying both the Nevanlinna-Pick property and hyponormality
of multiplication operators.
Nevertheless, they seem to indicate that the presence of both properties is special.

The aim of this note is to show that the Hardy space is essentially the only complete Nevanlinna-Pick space whose
multiplication operators are hyponormal. Recall that a reproducing kernel Hilbert space $\cH$ with kernel $K$ on a set $X$ is called
\emph{irreducible} if $K(x,y)$ is never zero
for $x,y \in X$ and if $K(\cdot,x)$ and $K(\cdot,y)$ are linearly independent for different $x,y \in X$.
We call a set $A \subset \mathbb D$ a \emph{set of uniqueness} for $H^2$ if the only element
of $H^2$ which vanishes on $A$ is the zero function.
The main result now reads as follows.

\begin{thm}
  \label{T:main}
  Let $\cH$ be an irreducible complete Nevanlinna-Pick space on a set $X$ with kernel $K$
  such that all
  multiplication operators on $\cH$ are hyponormal. Then one of the following possibilities holds:
  \begin{enumerate}[label=\normalfont{(\arabic*)}]
    \item $X$ is a singleton and $\cH = \mathbb C$.
    \item There is a set of uniqueness $A \subset \mathbb D$ for $H^2$,
      a bijection $j: X \to A$ and a nowhere vanishing function $\delta:X \to \mathbb C$
      such that
      \begin{equation*}
        K(\lambda,\mu) = \delta(\lambda) \overline{\delta(\mu)} \, k(j(\lambda),j(\mu)),
      \end{equation*}
      where $k(z,w) = (1 - z \overline{w})^{-1}$ denotes the Szeg\H{o} kernel.
      Hence,
      \begin{equation*}
        H^2 \to \cH, \quad f \mapsto \delta (f \circ j),
      \end{equation*}
      is a unitary operator. If $X$ is endowed with a topology such that $K$ is separately continuous on $X \times X$,
      then $j$ is continuous. If $X \subset \bC^n$ and $K$ is holomorphic in the first variable,
      then $j$ is holomorphic.
  \end{enumerate}
\end{thm}

Since the Hardy space $H^2$ is a complete Nevanlinna-Pick space whose multiplication operators
are hyponormal, it is easy to see
that the same is true for every space as in part (2).
Hence, this result characterizes Hilbert function spaces with these two properties.

\begin{rem}
  (a)
  It is well known that sets of uniqueness for $H^2$ are characterized by the Blaschke condition (see, for example,
  \cite[Section II 2]{Garnett07}):
  A set $A \subset \bD$ is a set of uniqueness for $H^2$ if and only if
  \begin{equation*}
    \sum_{a \in A} (1 - |a|) = \infty.
  \end{equation*}
  (b)
    The condition that $K(x,y)$ is never zero is not very restrictive. Indeed, if we drop this condition,
    then $X$ can be partitioned into sets $(X_i)$ such that the restriction of $\cH$ to each $X_i$ (compare the next section)
    is an irreducible complete Nevanlinna-Pick space (see \cite[Lemma 7.2]{AM02}).
    This yields a decomposition of $\cH$ into an orthogonal direct
    sum of irreducible complete Nevanlinna-Pick spaces $\cH_i$.
    It is not hard to see that this decomposition is reducing for multiplication operators.
    Hence, all multiplication operators on $\cH$ are hyponormal if and only if
    this is true for each summand.
    We omit the details.
\end{rem}

Before we come to the proof of the main result, let us consider an application to Hilbert function spaces
in higher dimensions. In particular, this applies to holomorphic Hilbert function spaces on the open unit ball
in $\bC^n$ for $n \ge 2$.
Standard examples of such spaces either have the property that all multiplication
operators are hyponormal (such as Hardy and Bergman space)
or have the Nevanlinna-Pick property (such as the Drury-Arveson space, see the next section), but not both.
This is not a coincidence.

\begin{cor}
  Let $n \ge 3$ be a natural number, and let $U \subset \bR^n$ be an open set. Then there is no irreducible
  complete Nevanlinna-Pick space on $U$ which consists of continuous functions and whose
  multiplication operators are all hyponormal.
\end{cor}

\begin{proof}
  Assume toward a contradiction that $\cH$ is such a Hilbert functions space, and let $K$ be its kernel. Since
  the functions in $\cH$ are continuous, it follows that $K$ is separately continuous. Hence, Theorem \ref{T:main}
  implies that there is a continuous injection $j: U \to \mathbb D$. But this is impossible if $n \ge 3$ 
  due to Brouwer's domain invariance theorem \cite{Brouwer11}.
\end{proof}

\section{Embedding into Drury-Arveson space}
\label{S:prelim}

As a first step in the proof of the main result, we will embed the complete Nevanlinna-Pick
space $\cH$ into the Drury-Arveson space.
Given a cardinal $d$, we write $\bB_d$ for the open unit ball in $\ell^2(d)$. The Drury-Arveson
space $H^2_d$ is the reproducing kernel Hilbert space on $\bB_d$ with kernel
\begin{equation*}
  k_d(z,w) = \frac{1}{1-\langle z,w \rangle}.
\end{equation*}
If $d=1$, this is the Hardy space $H^2$. For $d \ge 2$,
Arveson \cite{Arveson98} exhibited multipliers on $H^2_d$ which are not hyponormal by
showing that their spectral radius is strictly less than their multiplier norm. Indeed, if $z_1$ and $z_2$ denote
the coordinate functions on $\bC^2$, then $M_{z_1 z_2}$ is not
hyponormal on $H^2_2$, as
\begin{equation*}
  ||M_{z_1 z_2} z_1 z_2 ||^2 = \frac{1}{6} < \frac{1}{4} = || M_{z_1 z_2}^* z_1 z_2 ||^2
\end{equation*}
(see \cite[Lemma 3.8]{Arveson98}). This observation readily generalizes to $d \ge 2$.

Given a subset $Y \subset \bB_d$, we write
$H^2_d \res{Y}$ for the reproducing kernel Hilbert space on $Y$ with kernel
$k_d \res{Y \times Y}$. A well-known result about Hilbert function spaces asserts that
\begin{equation*}
  H^2_d \res{Y} = \{f \res{Y}: f \in H^2_d\},
\end{equation*}
and that the restriction map $f \mapsto f \res{Y}$ is a coisometry. Hence, if
\begin{equation*}
  I(Y) = \{f \in H^2_d: f \res{Y} = 0 \}
\end{equation*}
denotes the kernel of the restriction map, then
\begin{equation}
  \label{E:res_map}
   H^2_d \ominus I(Y) \to H^2_d \res{Y}, \quad f \mapsto f \res{Y},
\end{equation}
is a unitary. The following theorem due to
Agler and McCarthy provides the desired embedding of $\cH$ into the Drury-Arveson space.

\begin{thm}
  \label{T:embed_into_DA}
  Let $\cH$ be an irreducible complete Nevanlinna-Pick space on a set $X$ with kernel $K$.
  Assume that $K$ is normalized at $\lambda_0 \in X$ in
  the sense that $K(\lambda_0,\mu) = 1$ for all $\mu \in X$. Then
  there is a cardinal $d$ and an injection $b: X \to \bB_d$ with $b(\lambda_0)=0$ such that
  \begin{equation*}
    K(\lambda,\mu) = \frac{1}{1-\langle b(\lambda), b(\mu) \rangle} \quad (\lambda,\mu \in X).
  \end{equation*}
  Hence,
  \begin{equation*}
    H^2_d \ominus I(Y) \to \cH, \quad f \mapsto (f \res{Y}) \circ b,
  \end{equation*}
  is a unitary operator, where $Y = b(X)$.
\end{thm}

\begin{proof}
  See \cite{AM00}, Theorem 3.1 in \cite{BTV01}, or Theorem 8.2 and Theorem 7.31 in \cite{AM02}. To deduce the second part from
  the first one, note that the identity for the kernels implies that
  \begin{equation*}
    H^2_d \res{Y} \to \cH, \quad f \mapsto f \circ b,
  \end{equation*}
  is unitary. Therefore, the composition of this map with the unitary operator in \eqref{E:res_map}
  is unitary as well.
\end{proof}

In the above setting, let $\cF_Y = H^2_d \ominus I(Y)$. This space is co-invariant under multiplication operators.
Clearly, every $\varphi \in \Mult(H^2_d)$ restricts to a multiplier on $H^2_d \res{Y}$, and hence
gives rise to the multiplier $(\varphi \res{Y}) \circ b$ on $\cH$. If $U$ denotes the unitary operator
in Theorem \ref{T:embed_into_DA}, then
\begin{equation*}
  U^* M_{(\varphi |_Y \circ b)} U = P_{\cF_Y} M_\varphi \res{\cF_Y}.
\end{equation*}
Thus, if we assume that all multiplication operators on $\cH$ are hyponormal, then all operators appearing
on the right-hand side of the last identity are hyponormal as well. We will use this fact to show that $\cF_Y$
can be identified with $H^2$.

\section{Proof of the main result}
\label{S:proof}

The discussion at the end of the last section suggests studying compressions of multiplication
operators to co-invariant subspaces such that the compressed operator is hyponormal.
We need the following simple observation.

\begin{lem}
  \label{L:hyponormal_isometric}
  Let $\cH$ be a Hilbert space,
  let $T \in \cB(\cH)$ and let $M \subset \cH$ be a co-invariant subspace for $T$. Suppose
  that the compression of $T$ to $M$ is hyponormal. If $f \in M$ with $||T^* f|| = ||T f||$,
  then $T f \in M$.
\end{lem}

\begin{proof}
  Since $M$ is co-invariant under $T$, and since $P_M T \big|_M$ is hyponormal, we have
  \begin{equation*}
    ||T^* f|| \le ||P_M T f|| \le ||T f|| = ||T^* f||.
  \end{equation*}
  Consequently, $||P_M T f|| = ||T f||$, and hence $T f \in M$.
\end{proof}

We will apply this observation to multiplication operators on $H^2_d$. Since the coordinate functions $z_i$
are multipliers on $H^2_d$, it follows from unitary invariance of the Drury-Arveson space that
all functions of the form $\kerpart{w}$ for $w \in \ell^2(d)$ are multipliers on $H^2_d$.

\begin{lem}
  \label{L:hyponormal_DA_kernel}
  Suppose that $\cF \subset H^2_d$ is a closed subspace which is co-invariant under multiplication operators. Let $z \in \bB_d$, and suppose that the compression
  $P_{\cF} M_{\kerpart{z}} \big|_{\cF}$ is hyponormal. Then the following assertions hold.
  \begin{enumerate}[label=\normalfont{(\alph*)}]
    \item 
      If $1 \in \cF$ and
      $K(\cdot, z) \in \cF$, then $\kerpart{z} \in \cF$.
    \item
    If $\kerpart{z} \in \cF$, then
    $\kerpart{z}^n \in \cF$ for all $n \ge 1$.
  \end{enumerate}
\end{lem}

\begin{proof}
  (a)
  Clearly, we may assume that $z \neq 0$, and define $w = z / ||z||$. Then
  \begin{equation*}
    \iota: H^2 \to H^2_d, \quad \sum_{n=0}^\infty a_n \zeta^n \mapsto \sum_{n=0}^\infty a_n \kerpart{w}^n,
  \end{equation*}
  is an isometry, where $\zeta$ denotes the identity function on $\bC$.
  Under this embedding,
  the unilateral shift $M_\zeta$
  on $H^2$ corresponds to the restriction of $M_{\kerpart{w}}$ to the reducing subspace
  $\iota(H^2)$. In particular, $M_{\kerpart{w}} \res{\iota(H^2)}$ is an isometry.

  Now, consider
  \begin{equation*}
    f = K(\cdot,z) - 1 = \sum_{n=1}^\infty \kerpart{z}^n \in \cF.
  \end{equation*}
  Observe that $f$ is contained in the range
  of the isometry $M_{\kerpart{w}} \res{\iota(H^2)}$, hence
  \begin{equation*}
    || M_{\kerpart{w}} f|| = ||f|| = ||M_{\kerpart{w}}^* f ||.
  \end{equation*}
  Lemma \ref{L:hyponormal_isometric} implies that $\cF$ contains the element
  $M_{\kerpart{z}} f$, and thus also
  \begin{equation*}
    f - M_{\kerpart{z}} f = \kerpart{z} \in \cF.
  \end{equation*}

  (b)
  The proof is by induction on $n$. The base case $n=1$ holds by assumption. Suppose that $n \ge 2$ and the assertion
  is true for $n-1$. The same argument as in the proof of part (a), applied to $\kerpart{z}^{n-1}$ in place of $f$, shows that
  \begin{equation*}
    ||M_{\kerpart{z}} \kerpart{z}^{n-1}|| = ||M_{\kerpart{z}}^* \kerpart{z}^{n-1}||,
  \end{equation*}
  so that
  \begin{equation*}
    \kerpart{z}^n = M_{\kerpart{z}} \kerpart{z}^{n-1} \in \cF
  \end{equation*}
  by Lemma \ref{L:hyponormal_isometric}.
\end{proof}

Given $Y \subset \bB_d$, it can happen that there is a larger set $Z \supset Y$ such that every function in $H^2_d \res{Y}$
extends uniquely to a function in $H^2_d \res{Z}$. To account for that, we define
\begin{equation*}
  \overline{Y} = \{z \in \bB_d: f(z) = 0 \text{ for all } f \in I(Y) \}.
\end{equation*}
Then $\overline{Y}$ is the largest set which contains $Y$ and satisfies this extension property. Moreover,
it is easy to see that
\begin{equation*}
  \overline{Y} = \{ z \in \bB_d: K(\cdot,z) \in H^2_d \ominus I(Y) \}.
\end{equation*}

\begin{lem}
  \label{L:F_ball}
  Let $Y \subset \bB_d$ be a set with $0 \in Y$, and set
  $\cF_Y = H^2_d \ominus I(Y)$. If the compression
  $P_{\cF_Y} M_{\kerpart{w}} \res{\cF_Y}$ is hyponormal for
  every $w \in \bB_d$, then $\overline Y$ is a complex ball, that is,
  \begin{equation*}
    \overline{Y} = M \cap \bB_d
  \end{equation*}
  for some closed subspace $M$ of $\ell^2(d)$.
\end{lem}

\begin{proof}
  Let $M$ be the closed linear span of $\overline{Y}$.
  Observe that for all $w \in \overline{Y}$, we have $K(\cdot,w) \in \cF_Y$. Since $1 = K(\cdot,0) \in \cF_Y$,
  part (a) of Lemma \ref{L:hyponormal_DA_kernel} implies that
  $\kerpart{w} \in \cF_Y$
  for all $w \in \overline{Y}$. It follows that
  \begin{equation*}
    \kerpart{v} \in \cF_Y \quad \tfa v \in M,
  \end{equation*}
  as $v \mapsto \kerpart{v}$ is a conjugate linear isometry.
  Using part (b) of Lemma \ref{L:hyponormal_DA_kernel}, we deduce that
  \begin{equation*}
    K(\cdot,v) = \sum_{n=0}^\infty \kerpart{v}^n \in \cF_Y
  \end{equation*}
  for all $v \in M \cap \bB_d$. This argument shows that $\overline{Y} \supset M \cap \mathbb B_d$, and the reverse
  inclusion is trivial.
\end{proof}

We can now prove the main result.

\begin{proof}[Proof of Theorem \ref{T:main}]
  If $X$ is a singleton, there is nothing to prove. Otherwise, fix $\lambda_0 \in X$. Since $K$ is an irreducible
  kernel, it is nowhere zero, so we can consider the normalized kernel defined by
  \begin{equation*}
    \widetilde K(\lambda,\mu) = \frac{K(\lambda,\mu)}{\delta(\lambda) \overline{\delta(\mu)}},
  \end{equation*}
  where
  \begin{equation*}
    \delta(\lambda) = \frac{K(\lambda,\lambda_0)}{\sqrt{K(\lambda_0,\lambda_0)}}.
  \end{equation*}
  Then $\widetilde K(\lambda_0,\mu) = 1$ for all $\mu \in X$. Moreover, if $\widetilde \cH$ denotes the reproducing kernel
  Hilbert space with kernel $\widetilde K$, then
  \begin{equation*}
    \widetilde \cH \to \cH, \quad f \mapsto \delta f
  \end{equation*}
  is a unitary operator.
  It is easy to see that $\widetilde \cH$ also
  satisfies the hypotheses of Theorem \ref{T:main}, so we will work with $\widetilde \cH$ instead of $\cH$.

  We will show that $\widetilde \cH$ can be identified with $H^2_{d'}$ for a suitable
  cardinal $d'$. It will then follow that $d'$ is necessarily $1$.
  By Theorem \ref{T:embed_into_DA}, there is an injection $b: X \to \mathbb B_d$ for some cardinal $d$
  such that $0 = b(\lambda_0) \in b(X)$ and such that
  \begin{equation*}
    \widetilde K(\lambda,\mu) = k_d(b(\lambda),b(\mu))
  \end{equation*}
  holds for all $\lambda,\mu \in X$.
  Define $Y = b(X)$ and $\cF_Y = H^2_d \ominus I(Y)$, and
  note that $0 \in Y$.
  The discussion at the end of Section \ref{S:prelim} now shows that $\cF_Y$
  satisfies the hypotheses of Lemma \ref{L:F_ball}, hence
  \begin{equation*}
    \overline{Y} = M \cap \bB_d
  \end{equation*}
  for some closed subspace $M$. Let $d'$ be the dimension of the Hilbert space $M$. As $X$ is not a singleton, $d' \neq 0$.
  Clearly, $\cF_Y = \cF_{\overline{Y}}$, so that the restriction map from
  $\cF_Y$ into $H^2_d \res{\overline {Y}}$ is unitary.
  If $V$ is an isometry from $\ell^2(d')$ onto $M \subset \ell^2(d)$, we have
  \begin{equation*}
    k_d(V(z), V(w)) = k_{d'} (z,w) \quad \tfa z,w \in \bB_{d'}.
  \end{equation*}
  Therefore,
  \begin{equation*}
    \cF_Y \to H^2_{d'}, \quad f \mapsto f \circ V,
  \end{equation*}
  is a unitary operator as well. Combining this map with the unitary from Theorem \ref{T:embed_into_DA},
  we obtain a unitary
  \begin{equation*}
    H^2_{d'} \to \widetilde \cH, \quad f \mapsto f \circ j,
  \end{equation*}
  where $j = V^* \circ b$.

  By assumption, all multiplication operators on $\widetilde \cH$ are hyponormal,
  hence the same is true for $H^2_{d'}$. This is only possible if
  $d' = 1$ (see the discussion at the beginning of Section \ref{S:prelim}), so that the last operator is in fact
  a unitary from $H^2$ onto $\widetilde \cH$. Injectivity of this operator implies that $A = j(X)$ is a set
  of uniqueness for $H^2$. Combining the identities for the various kernels, we see that
  \begin{equation}
    \label{E:kernel_id}
    K(\lambda,\mu) = \delta(\lambda) \overline{\delta(\mu)} k(j(\lambda), j(\mu)) \quad \tfa \lambda,\mu \in X,
  \end{equation}
  as asserted.

  To prove the additional assertion, let $\lambda_0 \neq \mu \in X$. Then $j(\mu) \neq 0$, so rearranging
  equation \eqref{E:kernel_id}, we obtain for $j$ the formula
  \begin{equation*}
    j(\lambda) = \big(\overline{j(\mu)}\big)^{-1} \Big( 1 - \frac{\delta(\lambda) \overline{\delta(\mu)}}{K(\lambda,\mu)} \Big).  
  \end{equation*}
  Taking the definition of $\delta$ into account, it follows that $j$ is continuous (respectively holomorphic)
  whenever $K(\cdot,\mu)$ is.
\end{proof}

\begin{rem}
  \label{R:equivalent_to_DA}
  (a) Since $d' = 1$ in the last proof, the isometry $V$ is of the form
  $\lambda \mapsto \lambda w$ for some unit vector $w$ in the one-dimensional space $M$.
  It is easy to see that in this situation, the inverse of the unitary
  \begin{equation*}
    \cF_{Y} = \cF_{\overline Y} \to H^2, \quad f \mapsto f \circ V,
  \end{equation*}
  is given by
  \begin{equation*}
    H^2 \to \cF_{Y} \subset H^2_d, \quad \sum_{n=0}^\infty a_n \zeta^n \mapsto \sum_{n=0}^\infty
    a_n \kerpart{w}^n.
  \end{equation*}
  An isometric embedding of this type was used in the proof of Lemma \ref{L:hyponormal_DA_kernel}.
  
  (b)
  For the most part of the proof of Theorem \ref{T:main}, we only used hyponormality of operators
  of the form $P_{\cF_Y} M_{\kerpart{w}} \res{\cF_Y}$ for $w \in \bB_d$ (notation as above).
  If $\cH$ is an irreducible complete Nevanlinna-Pick space with kernel $K$, normalized at some point $\lambda_0$, then
  these operators correspond to multiplication operators on $\cH$ with
  multipliers of the form
  \begin{equation}
    \label{E:NP_coordinates}
    \varphi(\cdot) = \langle b(\cdot), w \rangle \quad(w \in \bB_d),
  \end{equation}
  where $b$ is the injection from Theorem \ref{T:embed_into_DA}.
  These multipliers play the role of coordinate functions for Nevanlinna-Pick spaces (see the discussion
  preceding Beurling's theorem for Nevanlinna-Pick spaces \cite[Theorem 8.67]{AM02}).
  
  The only argument which requires hyponormality of more general multiplication operators is the proof that $d'=1$.
  Thus, if we weaken the hypothesis of Theorem \ref{T:main} and only require hyponormality of multiplication operators
  corresponding to functions as in \eqref{E:NP_coordinates},
  then $\cH$ will be equivalent to $H^2_{d'}$ (in the sense
  of part (2) of Theorem \ref{T:main}) for some cardinal $d'$.
\end{rem}

\begin{acknowledgments}
  The author would like to thank his advisor, Ken Davidson, for his advice and support.
\end{acknowledgments}

\bibliographystyle{amsplain}
\bibliography{literature}
\end{document}